\documentclass[a4paper]{amsart}

\usepackage{amssymb,amsbsy,amsmath,amsfonts,amssymb,amscd}
\usepackage{latexsym}
\usepackage{graphics}
\usepackage{color}
\input xy
\xyoption{all}

\theoremstyle{plain}
\newtheorem{thm}{Theorem}[section]
\newtheorem{theorem}[thm]{Theorem}

\newtheorem{lemma}[thm]{Lemma}

\newtheorem{proposition}[thm]{Proposition}
\theoremstyle{definition}
\newtheorem{remark}[thm]{Remark}

\newtheorem{defin}[thm]{Definition}

\newtheorem{assumption}[thm]{Assumption}

\newtheorem{example}[thm]{Example}

\numberwithin{equation}{section}

\newcommand{\sI}{{\mathcal I}}

\newcommand{\sN}{{\mathcal N}}

\def\eea{\end{eqnarray*}}
\def\bea{\begin{eqnarray*}}

\newcommand\dual{\mathrel{\raise3pt\hbox{$\underline{\mathrm{\thinspace d
\thinspace}}$}}}
\newcommand\qe{\ifhmode\unskip\nobreak\fi\quad $\Box$}       

\def\BOX{\hfill\lower.5\baselineskip\hbox{$\Box$}}

\newtheorem{theo}{Theorem}[section]

\newtheorem{remarkk}[theo]{Remark}

\DeclareMathOperator{\Aut}{Aut}

\title [Stable curves]{Cyclic covers of Stable curves and their moduli spaces}

\author{Binru Li }

\address {Shanghai Center for Mathematicial Science, Fudan University.\\
	220 Handan road, Yangpu District.\\
	200433 Shanghai.
	Tel: +86 55665384}
\email{binruli@fudan.edu.cn}



\thanks{The author is currently supported by Shanghai Center for Mathematical Science (SCMS), Fudan University. Part of this work took place in the framework of the ERC Advanced grant n. 340258, `TADMICAMT'.}

\begin{document}

\maketitle

\begin{abstract}	
We study the deformation of $G$-marked stable curves in the case where $G$ is a cyclic group, and construct a parameterizing space for $G$-marked stable curves of a given numerical type.

This is then used in order to study the components of the locus of stable curves admitting the action of a cyclic group of non prime order $d$,
extending work of F. Catanese in the case where $d$ is prime.
\end{abstract}
\section{Introduction}
The purpose of this article is to study the structure of the locus $(\overline{\mathfrak{M}_g}-\mathfrak{M}_g)(G)$ of (non-smooth) stable curves of genus $g$ inside the compactified moduli space $\overline{\mathfrak{M}_g}$ which admitting an effective action by a cyclic group $G$. 

In \cite{Cor87} and \cite{Cor08} M. Cornalba determined the irreducible components of $Sing(\mathfrak{M}_g)$, the singular locus of the moduli scheme of smooth projective curves of genus $g\geq 2$. The result was obtained by showing that $\mathfrak{M}_g(\mathbb{Z}/p)$, the locus inside $Sing(\mathfrak{M}_g)$ of curves admitting an effective action by a cyclic group of prime order $p$, is irreducible and maximal (i.e. being not contained in another locus) except for finitely many cases. The main ingredient Cornalba used is that the locus corresponding to cyclic covers of prime order of smooth curves with a fixed combinatorial datum, called the numerical type (see Definition \ref{smnumtype}), is an irreducible Zariski closed subset of the moduli space $\mathfrak{M}_g$. Catanese in \cite{Cat12} extended this result to the case of cyclic groups of any order (cf. Theorem \ref{smpara}).

The studies of such loci can be continued in two directions: 
 
In one direction more finite groups $G$ are considered. For instance the case where $G=D_n$, the dihedral group of order $2n$, was investigated in a series of papers by F. Catanese, M. L{\"o}nne and F. Perroni (cf. \cite{CLP11}, \cite{CLP15}) and later by B. Li and S. Weigl (cf. \cite{LW16}). The main difficulty there is that for general groups a numerical type might correspond to a reducible subset of the moduli space. In \cite{CLP15} the authors introduced a new homological invariant which enables them to distinguish the irreducible components asymptotically (i.e., when the genus of the quotient curve $>>0$).

The other direction is to consider the boundary of the compactified moduli space $\overline{\mathfrak{M}_g}$. In \cite{Cat12}, Catanese determined the irreducible components of $Sing(\overline{\mathfrak{M}_g}-\mathfrak{M}_g)$ by studying the loci $(\overline{\mathfrak{M}_g}-\mathfrak{M}_g)(\mathbb{Z}/p)$ and obtained analogous results as in the smooth case. In this case, the locus of stable curves with a given numerical type is not necessarily Zariski closed: if a stable curve $C_1$ is smoothable to another stable curve $C_2$, then the corresponding locus of $C_1$ is contained in the closure of that of $C_2$, hence one should look at the non-smoothable stable curves. Hence in the boundary case the notion of maximal means that the Zariski closure of the locus is maximal (cf. Definition \ref{stratum}).

In this article we go into both directions, studying the loci $(\overline{\mathfrak{M}_g}-\mathfrak{M}_g)(\mathbb{Z}/d)$ of non-prime order $d$ and generalize several results in \cite{Cat12}. 

\medskip

This article is organized as follows.

In section 1 we give the definition of a $G$-marked stable curve (i.e., a stable curve $C$ admitting an effective action by a finite group $G$, cf. Definition \ref{Gcurve}) and associated notions.

In section 2 we study the $G$-equivariant deformation (cf. Definition \ref{equidef}) of $G$-marked stable curves, and determine when a $G$-marked stable curve is $G$-equivariantly smoothable. Then we define the associated numerical type for $G$-marked stable curves and prove the main result of this section that, for $G$-marked stable curves with a given numerical type, there is a parameterizing space (cf. Theorem \ref{thm1}):
 \begin{theorem}\label{thm1}
 	Given a $G$-marked stable curve $(C,G,\rho)$, there exists a connected complex manifold $\mathcal{T}_{[\mathcal{D}(C,G,\rho)]}$ parameterizing all $G$-marked stable curves with numerical type $[\mathcal{D}(C,G,\rho)]$.\\
 	If moreover $(C,G,\rho)$ is $G$-equivariantly non-smoothable, denoting by $\mathfrak{M}_{[\mathcal{D}(C,G,\rho)]}$ the image of the natural morphism $\mathcal{T}_{[\mathcal{D}(C,G,\rho)]}\rightarrow \overline{\mathfrak{M}_g}-\mathfrak{M}_g $, then each point inside $\mathfrak{M}_{[\mathcal{D}(C,G,\rho)]}$ has finite inverse image in $\mathcal{T}_{[\mathcal{D}(C,G,\rho)]}$, and the closure $\overline{\mathfrak{M}_{[\mathcal{D}(C,G,\rho)]}}$ consists of $G$-marked stable curves which can be $G$-equivariantly deformed into a curve with numerical type $[\mathcal{D}(C,G,\rho)]$.  
 \end{theorem}

In section 3 we study the irreducible components of $(\overline{\mathfrak{M}_g}-\mathfrak{M}_g)(G)$ for the case $G=\mathbb{Z}/d$, the idea is to determine when a $G$-stratum (i.e., the image inside $(\overline{\mathfrak{M}_g}-\mathfrak{M}_g)(\mathbb{Z}/d)$ of the parameterizing space of a given numerical type, cf. Definition \ref{stratum}) is maximal. For this we need to compare all the order $d$ cyclic subgroups of the stratum (cf. Definition \ref{StrataAut}).\\ 
Due to some phenomena arising from the smooth case (cf. Proposition \ref{nonfull}), the automorphism group of a stratum might become very complicated, making it impossible to give a brief and explicit description for maximal strata. Hence we make some technical assumptions.
\begin{assumption}(cf. Assumption \ref{assum})
	\begin{itemize}
		\item[(0)] $(C=\sum_{i \in I}C_i,G,\rho)$ is $G$-equivariantly non-smoothable.
		\item[(1)] For a general stable curve $(C,G,\rho)$ in the stratum we have $H_i=Aut(\widetilde{C_i})$ and $g(\widetilde{C_i})\geq 2$ for all $i$.
		\item[(2)] For any $i\in I$, the parameterizing space $\mathcal{T}_{h_i,r_i}$ has dimension$>0$.
	\end{itemize}
\end{assumption}
With the above assumptions we prove the main result of this article (cf. Theorem \ref{mainthm}):
\begin{theorem}
	 Under the conditions of Assumption \ref{assum}, we have the following:
\begin{itemize}
	\item[(1)] For a $G$-equvariantly non-smoothable $G$-marked stable curve $(C=\sum_{i\in I}C_i,G,\rho)$, the induced stratum $\mathfrak{M}_{C'}$, where $C'=C/G$, is maximal iff for a general stable curve (by abuse of notation we denote still by) $(C,G,\rho)$ in the stratum:
	\begin{itemize}
		\item[(a)] The cases in Lemma \ref{case 1} do not occur.
		\item[(b)]  For any $\beta\in \Aut(C)$ (of order $d$) and any node $p$ where Case (II-i) happens, the following holds:
		$$\zeta_{p,1}^{b(p,1)c(\beta,p,1)}\zeta_{p,2}^{b(p,2)c(\beta,p,2)}\neq 1.$$
	\end{itemize}
    \item[(2)] The Zariski closure of each maximal stratum in $(1)$ is an irreducible component of $(\overline{\mathfrak{M}_g}-\mathfrak{M}_g)(G)$.
\end{itemize}
\end{theorem}

\section{Notation}

Let $C$ be a (non-smooth) stable curve (i.e., $C$ has at most nodes as singularities and $Aut(C)$, the automorphism group of $C$, is finite), 
$$ C = \sum_{i \in I} C_i .$$

We define $\sI$ to be the graph whose set of vertices is the set $I$,
and whose set of edges is the set $\sN$ of  the nodes $P \in C$. 

We let $\sN_i : = \sN \cap C_i$, i.e., these are the edges of the graph containing the vertex  $i$. 

Note that if $P \in \sN, P  \in C_i \cap C_j, i\neq j,$
then $P$ yields  an edge connecting two distinct vertices,
else, if $ P \in C_i $ and $ P \notin C_j, \ \forall j \neq i$,  $P$ yields a loop based at $i$. Hence we have $$\mathcal{N}_i=\mathcal{N}^{(1)}_i\cup\mathcal{N}^{(2)}_i,$$
where $\mathcal{N}^{(1)}_i$ corresponds to edges connecting two distinct vertices (one of the vertex is $i$) and $\mathcal{N}^{(2)}_i$ corresponds to loops based at $i$.

Set further $ C- C_i =  \overline{C \setminus C_i}$. 

\begin{defin}\label{Gcurve}
\begin{itemize}
\item[(1)] Let $G$ be a finite group. A {\em $G$-marked stable curve} is a triple $(C,G,\rho)$, where $C$ is a stable curve, $\rho:G\hookrightarrow Aut (C)$ is an injective homomorphism, i.e., $G$ acts effectively on the stable curve $C$. When $\rho$ is clear, for instance if $G$ is a subgroup of $Aut(C)$, we write for short $(C,G)$ instead of $(C,G,\rho)$.
\item[(2)] We call $(C,G,\rho)$ a smooth (resp. irreducible) $G$-marked curve if $C$ is smooth (resp. irreducible).
	  \end{itemize}
\end{defin}
\begin{remark}
	In the case where $\rho$ is clear from the context, we identify $G$ with its image $\rho(G)$ and write $G\subset Aut(C)$.
\end{remark}
Given a $G$-marked curve $(C,G,\rho)$, then $G$ acts naturally on the graph $\sI$, and on the set $I$.

\begin{defin}
i) Let $K_v $ be the kernel of the action on $I$, and let instead $G_i$ be the stabilizer of $i \in I$; 
in other words, $$G_i := \{ g | g (C_i) = C_i \}$$ and $K_v = \cap_{i\in I} G_i$.

ii) Let $K$ be the kernel of the action on $\sI$, and let, for $ P \in \sN$, $G_P$ be the stabilizer of $P$;
hence $K = K_v \cap (\cap_{P \in \sN} G_P)$. We let moreover $G'_i$ be the subgroup of $G_i$ which
fixes the nodes in $\sN_i$, the nodes of $C$ belonging to $C_i$, and we let $G''_i$ be the subgroup which acts trivially on $C_i$.
Hence $ K = \cap_{i\in I}G'_i$. We denote by $n_i$ the order of $G/G_i$.

iii) In the case $G$ is an abelian group, let  $H_i$ be the quotient group $G_i / G''_i$, respectively $H_i' : = G'_i / G''_i$. Necessarily $H'_i$
is a cyclic subgroup. We denote by $d_i$ (resp. $d'_i$) the order of $H_i$ (resp. $H'_i$).

iv) Setting where $I_0=\{i\in I|G=G''_i \}$, $I_1=\{i\in I|G=G_i\}$ and $I_2=\{i\in I|G\neq G_i\}$, then the set $I$ has a natural partion $I=I_0\cup I_1\cup I_2$.
\end{defin}


In the rest of this article $G$ shall denote a cyclic group $\mathbb{Z}/d$ with generator $\gamma$ and $\zeta_{d}:=\exp(\frac{2\pi \sqrt{-1}}{d})$. We work over the field of complex numbers $\mathbb{C}$. 

\section{Parameterizing space of cyclic coverings}
In this section we will construct parameterizing spaces for $G$-marked stable curves, first we review the case of smooth $G$-marked curves.\\

Let $(C,G)$ be a smooth irreducible $G$-marked curve. The action of $G$ on $C$ induces a (ramified) covering map $C\rightarrow C':=C/G$. For any $1\leq i\leq d-1$, we set
$$D_i:=\{p\in C'|p\text{ is a branching point, }\forall q\in C\ \text{over}\ p, Stab(q)=\langle \gamma^{gcd(i,d)}\rangle,$$
$$\text{ locally around }q \text{ the action of }\gamma^{gcd(i,d)} \text{ is given by } z\mapsto \zeta_d^{d/gcd(i,d)}z,   \}$$

\begin{defin}[\cite{Cat12}, Definition 2.2]\label{smnumtype}
	Let $C$ be a smooth irreducible projective curve of genus $g$ on which $G=\mathbb{Z}/d$ acts faithfully, and set $C'=C/G$, $h:=genus(C')$.\\
	Denote by $k_i=deg(D_i)$ for $i=1,...,d-1$, and by $(k_1,...,k_{d-1})$ the {\em branching sequence} of $\gamma$. A change of generator of $\mathbb{Z}/d$ corresponds to a $(\mathbb{Z}/d)^*$-action on the set of sequences, we denote the resulting equivalence class by $[(k_1,...,k_{d-1})]$, and call it the {\em numerical type} of the cyclic cover $C\rightarrow C'$.
\end{defin}
\begin{defin}[\cite{Cat12}, Definition 2.3]\label{adm}
	Given a branching datum corresponding to a sequence $[(k_1,...,k_{d-1})]$, set
	$$h:=1+\frac{2(g-1)}{2d}-\frac{1}{2}\sum_{i=1}^{d-1}k_i(1-\frac{\gcd(i,d)}{d}).$$
	The branching datum is said to be {\em admissible} for $d$ and $g$ if the following two conditions are satisfied:\\
	(1) $\sum_{i=1}^{d-1} k_ii\equiv 0 ~($mod $d)$, \\
	(2) $h$ is a positive integer; or $h=0$, $\gcd\{d,\gcd \{i|k_i\neq 0\}\}=1$.
\end{defin}
The main result for the parameterizing space of smooth $G$-marked curves is the following:
\begin{thm}[\cite{Cat12}, Theorem 2.4]\label{smpara}
The paris $(C,G)$, where $C$ is a complex projective curve of genus $g\geq 2$, and $G$ is a finite cyclic group of order $d$ acting faithfully on $C$ with a given branching datum $[(k_1,...,k_{d-1})]$ are parameterizing by a connected complex manifold $\mathcal{T}_{g,d;[(k_1,...,k_{d-1})]}$ of dimension $3(h-1)+k$, where $k:=\sum_ik_i$.\\
The image $\mathfrak{M}_{g,d;[(k_1,...,k_{d-1})]}$ of $\mathcal{T}_{g,d;[(k_1,...,k_{d-1})]}$ inside the moduli space $\mathfrak{M}_g$ is a closed subset of the same dimension $3(h-1)+k$.
\end{thm}
We will give an analogous result for $G$-marked stable curves. 
\begin{defin}\label{equidef}
	Let $(C,G,\rho)$ be a $G$-marked stable curve: a $G$-{\em equivariant deformation} of $(C,G,\rho)$ is a triple $(p:\mathfrak{C}\rightarrow B, G, \eta)$ such that
	\begin{itemize}
		\item[(1)] $p:\mathfrak{C}\rightarrow B$ is a deformation of $C$ over an irreducible base $B$ with all fibres stable curves and the central fibre $\mathfrak{C}_O\simeq C$ ($O\in B$). 
		\item[(2)] $\eta:G\rightarrow Aut(\mathfrak{C})$ is an injective homomorphism inducing an effective action on $\mathfrak{C}$ such that $p$ is $G$-equivariant (where the action of $G$ on $B$ is trivial) and $\eta|_{\mathfrak{C}_O}\simeq \rho$. 
	\end{itemize}
\end{defin}
\begin{defin}
	We say that a $G$-marked stable curve $(C,G,\rho)$ is $G$-{\em equivariantly non-smoothable} (or has no $G$-{\em equivariant smoothing}) if $(C,G,\rho)$ can not be $G$-equivariantly deformed to $(C',G,\rho')$ such that $C'$ has less nodes than $C$.
\end{defin}
We have the following criterion which tells when a $G$-marked stable curves is $G$-{\em equivariantly non-smoothable}, and generalizes the prime case in \cite{Cat12}, Lemma 4.3:
\begin{proposition}\label{nonsm}
	
		Let $P\in C=\sum_iC_i$ be a node, set $G_P:=Stab(P)$ the stabilizer group of $P$ in $G$, then the following are equivalent:\\
		(1) All points in $G(P)$(:= the orbit of $P$) can by simultaneously $G$-equivariantly smoothed.\\
		(2) The induced group homomorphism $G_P\rightarrow GL(\mathcal{E}xt^1(\Omega_C,\mathcal{O}_C)_P)\simeq \mathbb{C}^*$ is trivial.
			
		\end{proposition}
			\begin{proof}
			Recall the local to global spectral sequence:
			$$(*)\  0\rightarrow \bigoplus_iH^1(\Theta_{C_i}(-\sum_{j\neq i}(C_i\cap C_j)))\rightarrow Ext^1(\Omega^1_C,\mathcal{O}_C)\rightarrow \bigoplus_{P\in  \mathcal{N}}\mathcal{E}xt^1(\Omega^1_C,\mathcal{O}_C)_P\rightarrow 0,$$
			where $\Omega^1_C$ is the dualizing sheaf of $C$ and $\Theta_{C_i}$ denotes the tangent sheaf of $C_i$.\\
			Since $G$ is a cyclic group, $(*)$ remains exact after taking the subspaces of $G$-invariant vectors. Hence we have a surjection:
			$$Ext^1(\Omega_C,\mathcal{O}_C)^G\twoheadrightarrow (\bigoplus_{P\in \mathcal{N}}\mathcal{E}xt^1(\Omega^1_C,\mathcal{O}_C)_P)^G.$$
		Since $G_P$ is a subgroup of $G=\mathbb{Z}/d$, we have $G_P\simeq \mathbb{Z}/m$ for some $m|d$, define $r:=m/d$. 
	Denoting by $\bar{\gamma}$ the image of $\gamma$ in $G/G_P$, clearly $\bar{\gamma}$ is a generator of $G/G_P$. Up to a change of index, we can assume that $G(P)=\{P_1=P,...,P_r\}$, such that $\bar{\gamma}(P_i)=P_{i+1}$. An easy observation is that $G(P)$ can be simultaneously smoothed $\Leftrightarrow$ $\exists v=(v_1,...,v_r) \in (\oplus_{P_i\in G(P)}\mathcal{E}xt^1(\Omega_C,\mathcal{O}_C)_{P_i})^G$, such that $v_i\neq 0,1\leq i\leq r$.\\
		Assume $\gamma(v_1,...v_r)=(\lambda_rv_r,\lambda_1v_1,...\lambda_{r-1}v_{r-1})$ for some $\lambda_i\in \mathbb{C}^*$. It is easy to see that $\gamma^rv=v\Leftrightarrow \prod^r_{i=1}\lambda_i=1$. \\
		If there exists $0\neq v\in (\oplus_{P_i\in G(P)}\mathcal{E}xt^1(\Omega_C,\mathcal{O}_C)_{P_i})^G$, then we have $\gamma^rv=v$. Noting that $\gamma^r$ is a generator of $G_P$, the induced homomorphism $G_P\rightarrow \mathbb{C}^*$ is then given by $\gamma^r\mapsto \prod\lambda_i=1$, hence trivial. Conversely, if $G_P\rightarrow \mathbb{C}^*$ is trivial, let $v=(v_1,\lambda_1v_1,...,(\prod^{r-1}_{i=1}\lambda_i)v_1)$ with $v_1\neq 0$, it is clear that $\gamma v=v$, since $\prod\lambda_i=1$.
	\end{proof}

\begin{defin}\label{numtyp}
A $G$-marked stable curve $(C,G,\rho)$ has the following associated combinatorial datum $\mathcal{D}(C,G,\rho)$:
\begin{itemize}
	\item[(1)] A $G$-marked graph $(\mathcal{I},G,\tilde{\rho})$, i.e., the graph $\mathcal{I}$ with induced $G$-action $\tilde{\rho}$ from the action $\rho:G\rightarrow Aut(C)$. 
	\item[(2)] For any $i\in I$, recall that $H_i=G_i/G''_i$ and $d_i=Ord(H_i)$. The image of an element $\beta\in G_i$ in $H_i$ is denoted as $\bar{\beta}$. We get a $H_i$-marked curve $(C_i,H_i,\rho_i)$, denote by $\widetilde{C_i}$ the normalization of $C_i$ and set $g_i=genus(\widetilde{C_i})$, $h_i=$ genus of $\widetilde{C'_i}:=\widetilde{C_i}/H_i$.  The element $\overline{\gamma^{n_i}}$ generates $H_i$, hence it induces a branching sequence $(k_1(i),...,k_{d_i-1}(i))$ on $\widetilde{C_i}$. We denote by $\mathcal{R}'_i$ the set of ramification points on $\widetilde{C_i}$ which do not come from $\mathcal{N}_i$. 
\end{itemize} 
The automorphism group $Aut(G)=(\mathbb{Z}/d)^*$ acts naturally on the set of data $\{\mathcal{D}(C,G,\rho)\}$,  we call the resulting equivalence class, $[\mathcal{D}(C,G,\rho)]$ the {\em numerical type} of $(C,G,\rho)$.
\end{defin}
\begin{remark}
	\begin{itemize}
		\item[(1)] 	As in the smooth case, we can determine an ``admissible condition" for the above combinatorial data (for the case $G$ has a prime order, see \cite{Cat12}, Definition 4.8), which we will not use in later discussion.
		\item[(2)] For the $H_i$-marked curve $(\widetilde{C_i},H_i,\rho_i)$, it is important to consider the branching sequence $(k_1(i),...,k_{d_i-1}(i))$ (induced by $\overline{\gamma^{n_i}}$) instead of the equivalent $H_i$-class $[(k_1(i),...,k_{d_i-1}(i))]$. Later we will see the differences. 
		\item[(3)] For a non-smoothable $G$-marked curve, using Proposition \ref{nonsm}, we see that $\forall i\in I_0$, the component $C_i$ is smooth (i.e. $\mathcal{N}^{(2)}_i=\emptyset$).
	\end{itemize}
\end{remark}
Now we come to the main result of this section, which is a partial generalization of \cite{Cat12}, Theorem 4.10.\\
We denote by $Orb$ the set of $G$-orbits in $I$, for any $o\in Orb$, we define a subcurve of $C$ consisting of all components in the orbit $o$, $$C(o):=\bigcup_{i\in o}C_i.$$ 
We have an induced $G_o:=G/G''_i$-marked (nodal)-curve $(C(o),G_o,\rho_o)$ (note that $C(o)$ might be disconnected). The following lemma shows that we have a ``canonical form" for $(C(o),G_o,\rho_o)$.
\begin{lemma}\label{canact}
  The $G_o$-marked curve $(C(o),G_o,\rho_o)$ is $G_o$-equivariantly isomorphic to the canonical form $(\cup^{n_o}_{j=1}C^{(j)}_o, G_o,\tilde{\rho}_o)$, where $n_o=\#|o|(=n_i)$, $C^{(j)}_o$ are $n_o$ copies of an irreducible component $C_i$ in $C(o)$,  and $\cup^{n_o}_{j=1}C^{(j)}_o$ is a quotient of $\bigsqcup^{n_o}_{j=1}C^{(j)}_o$ by identifying a finite set of (unordered) pair of points $\mathcal{P}_o$ and $\tilde{\rho}_o$ is determined by the following morphisms:
	$$id:C^{(j)}_o\rightarrow C^{(j+1)}_o\ \forall 1\leq  j \leq n_o-1;\ \overline{\gamma^{n_o}}:C^{(n_o)}_o\rightarrow C^{(1)}_o.$$
\end{lemma}
\begin{proof}
	It is clear that the morphisms given in the lemma define an action of $G_o$ on $\bigsqcup^{n_o}_{j=1}C^{(j)}_o$. It is easy to check that the morphisms $\overline{\gamma^{j-1}}|_{C_i}:C^{(j)}_o(=C_i)\rightarrow \overline{\gamma^{j-1}}(C_i)$ induce a surjective $G_o$-equivariant morphism $\bigsqcup^{n_o}_{j=1}C^{(j)}_o\rightarrow C(o)$. Denoting by $\mathcal{P}_o$ the set of inverse images of nodes in $C(o)$ which do not have two branches on the same irreducible curve, we obtain a quotient curve $\cup^{n_o}_{j=1}C^{(j)}_o$ by identifying the pairs of points lying in the same inverse image in $\mathcal{P}_o$, then we have a $G_o$-equivariant isomorphism $\phi_o:\cup^{n_o}_{j=1}C^{(j)}_o\rightarrow C(o)$.
\end{proof}

\begin{theorem}\label{thm1}
	Given a $G$-marked stable curve $(C,G,\rho)$, there exists a connected complex manifold $\mathcal{T}_{[\mathcal{D}(C,G,\rho)]}$ parameterizing all $G$-marked stable curves with numerical type $[\mathcal{D}(C,G,\rho)]$.\\
	 If moreover $(C,G,\rho)$ is $G$-equivariantly non-smoothable, denoting by $\mathfrak{M}_{[\mathcal{D}(C,G,\rho)]}$ the image set of the natural morphism $\mathcal{T}_{[\mathcal{D}(C,G,\rho)]}\rightarrow \overline{\mathfrak{M}_g}-\mathfrak{M}_g $, then each point inside $\mathfrak{M}_{[\mathcal{D}(C,G,\rho)]}$ has finite inverse image in $\mathcal{T}_{[\mathcal{D}(C,G,\rho)]}$, and the closure $\overline{\mathfrak{M}_{[\mathcal{D}(C,G,\rho)]}}$ consists of $G$-marked stable curves which can be $G$-equivariantly deformed into a curve with numerical type $[\mathcal{D}(C,G,\rho)]$.  
\end{theorem}

\begin{proof}
     $\mathcal{T}_{[\mathcal{D}(C,G,\rho)]}$ is a product of three products of Teichm{\"u}ller spaces, corresponding to the partition $I=I_0\bigcup I_1\bigcup I_2$:
     $$\mathcal{T}_0:=\prod_{i\in I_0}\mathcal{T}_{h_i,r_i},$$
     where $r_i=\#|\mathcal{N}^{(1)}_i|+2\#|\mathcal{N}^{(2)}_i|$. Over each  $\mathcal{T}_{h_i,r_i}$ we have a family of curves of genus $h_i=g_i$ with 
     $r_i$ marked points.
     $$\mathcal{T}_1:=\prod_{i\in I_1}\mathcal{T}_{h_i,r_i}, $$
     where $r_i=\sum^{d_i-1}_{l=1}k_l(i)$ is the number of ramification points of the covering $\widetilde{C_i}\rightarrow \widetilde{C_i}/H_i$. Over each $\mathcal{T}_{h_i,r_i} $ we have a family of $H_i$-marked curves of genus $g_i$ with the branching sequence $(k_1(i),...,k_{d_i-1}(i))$ with respect to a fixed generator $\gamma_i:=\bar{\gamma}$ of $H_i$ (See \cite{Cat12}, Theorem 2.4). 
     $$\mathcal{T}_2:=\prod_{[i]\in \bar{I}_2}\mathcal{T}_{h_i,r_i},$$ 
     where $\bar{I_2}$ is the set of orbits in $I_2$, $r_i=\sum^{d_i-1}_{l=1}k_l(i)$. In each orbit $[i]$ we pick one $\mathcal{T}_{h_i,r_i} $, over which we construct a family of  $n_i$ disjoint copies of $H_i$-marked curves of genus $g_i$ with the branching datum $(k_1(i),...,k_{d_i-1}(i))$ with respect to a fixed generator $\gamma_i:=\overline{\gamma^{n_i}}$ of $H_i$.\\
     Define $$\mathcal{T}_{[\mathcal{D}(C,G,\rho)]}:=\mathcal{T}_1\times\mathcal{T}_2\times\mathcal{T}_3.$$
     Now we can glue the pull back of the families over each factor, by identifying the sections according to the numerical type $[\mathcal{D}(C,G,\rho)]$, to get a family $\mathcal{C}_{[\mathcal{D}(C,G,\rho)]}$ over $\mathcal{T}_{[\mathcal{D}(C,G,\rho)]}$.
    
     Each fibre of $\mathcal{C}_{[\mathcal{D}(C,G,\rho)]}$ is a stable curve, on which we will define an action of $G$, making it a $G$-marked stable curve with numerical type $[\mathcal{D}(C,G,\rho)]$. 
     
We pick a fibre $C=\sum_{i\in I}C_i$, first the numerical type $[\mathcal{D}(C,G,\rho)]$ gives a $G$-action on the set of curves and nodes, in order to define an action of $G$ on the curve, it suffices to define the action on each orbit of the curves:

     If $i\in I_0$, $\gamma$ acts trivially.
     
     If $i\in I_1$, we have a natural action of $H_i$ on $C_i$ which is induced by the branching datum $(k_1(i),...,k_{d_i-1}(i))$ with respect to $\gamma_i$, the chosen generator of $H_i$ (by abuse of notation, the corresponding automorphism is also denoted as $\gamma_i$). Then the action of $G$ is defined by the homomorphism $G\rightarrow H_i$ which sends $\gamma$ to $\gamma_i$.
         
     If $i\in I_2$, we have to define the action of $G$ on $C([i])$. First we have the action of $H_i$ on $C_i$ which is determined by the branching datum $(k_1(i),...,k_{d_i-1}(i))$ with respect to $\gamma_i$. The action of $G$, equivalently the automorphism corresponding to $\gamma$, is defined as follows:\\
      $\gamma:C_{\gamma^{l-1}(i)}\rightarrow C_{\gamma^l (i)},x\mapsto x$ if $1\leq l\leq n_{[i]}-1$,\\
      $\gamma:C_{\gamma^{n_{[i]}-1}(i)}\rightarrow C_i,x\mapsto \gamma_i x$.\\
       By Lemma \ref{canact}, this should be the expected action.
     
     If $(C,G,\rho)$ is $G$-equivariantly non-smoothable, our parameterizing space has the expected maximal dimension. By Proposition \ref{nonsm} we have that 
     $$(\bigoplus_{p\in \mathcal{N}}\mathcal{E}xt^1(\Omega^1_C,\mathcal{O}_C)_p)^G=0.$$ Taking the $G$-invariant subspaces of $(*)$, we get
     $$(\bigoplus_iH^1(\Theta_{C_i}(-\sum_{j\neq i}(C_i\cap C_j)))^G\simeq Ext^1(\Omega^1_C,\mathcal{O}_C)^G.$$
     It is easy to see that 
     $$(\bigoplus_iH^1(\Theta_{C_i}(-\sum_{j\neq i}(C_i\cap C_j)))^G=\bigoplus_{o\in Orb}(\bigoplus_{i\in o}H^1(\Theta_{C_i}(-\sum_{j\neq i}(C_i\cap C_j)))^G.$$ 
     
     For each $i\in I_0$, it is clear that $$H^1(\Theta_{C_i}(-\sum_{j\neq i}C_i\cap C_j))^G=H^1(\Theta_{C_i}(-\sum_{j\neq i}C_i\cap C_j)),$$ hence has dimension equal to $dim_{\mathbb{C}}\mathcal{T}_{g_i,r_i}$.
     
     Using a result of Pardini (cf. \cite{Par91}, Theorem 4.2), one can show easily that $\forall i\in I_1$, $$dim(H^1(\Theta_{C_i}(-\sum_{j\neq i}C_i\cap C_j))^G)=dimH^1(\Theta_{\widetilde{C'_i}}(-B_i))=dim\mathcal{T}_{h_i,r_i},$$ where $B_i$ is the branching locus of the covering $\widetilde{C_i}\rightarrow \widetilde{C'_i}$.
     
     For each $i\in I_2$, consider the map 
     $$H^1(\Theta_{C_i}(-\sum_{j\neq i}C_i\cap C_j))\rightarrow \bigoplus_{j\in [i]}H^1(\Theta_{C_j}(-\sum_{l\neq j}C_j\cap C_l)):$$ $$v\mapsto (v,\gamma(v),...,\gamma^{n_{[i]}-1}(v)).$$ It is easy to see that this induces an isomorphism between the subspaces 
     $$H^1(\Theta_{C_i}(-\sum_{j\neq i}C_i\cap C_j))^{H_i}\simeq (\bigoplus_{j\in [i]}H^1(\Theta_{C_j}(-\sum_{l\neq j}C_j\cap C_l)))^G.$$ Therefore we obtain that $$dim(\bigoplus_{j\in [i]}H^1(\Theta_{C_j}(-\sum_{l\neq j}C_j\cap C_l)))^G=dimH^1(\Theta_{C_i}(-\sum_{j\neq i}C_i\cap C_j))^{H_i}=dim\mathcal{T}_{h_i,r_i}.$$
     We see that the family $\mathcal{T}_{[\mathcal{D}(C,G,\rho)]}$ has the same dimension as the Kuranishi space $Ext^1(\Omega^1_C,\mathcal{O}_C)^G$.\\ 
      The finiteness of the morphism $\mathcal{T}_{[\mathcal{D}(C,G,\rho)]}\rightarrow \mathfrak{M}_{[\mathcal{D}(C,G,\rho)]} $ follows from the finiteness of the smooth case and the fact that the automorphism group of a stable curve is finite. The rest of the theorem follows from the definition of a $G$-equivariantly non-smoothable $G$-marked stable curve.
\end{proof}
\begin{defin}\label{stratum}
\begin{itemize}
    \item[(1)] We call the image of the natural map $\mathcal{T}_{[\mathcal{D}(C,G,\rho)]}\rightarrow \overline{\mathfrak{M}_g}-\mathfrak{M}_g$ a {\em stratum} with numerical type $[\mathcal{D}(C,G,\rho)]$, which we denote by $\mathfrak{M}_{[\mathcal{D}(C,G,\rho)]}$.
    \item[(2)] A stratum $\mathfrak{M}_{[\mathcal{D}(C,G,\rho)]}$ is called {\em maximal}, if it is not contained in the Zariski closure of another stratum $\overline{\mathfrak{M}_{[\mathcal{D}(C',G,\rho')]}}$, such that\\ $dim\mathfrak{M}_{[\mathcal{D}(C,G,\rho)]}<dim\mathfrak{M}_{[\mathcal{D}(C',G,\rho')]}$.
\end{itemize}
\end{defin}
It is clear that $(\overline{\mathfrak{M}_g}-\mathfrak{M}_g)(G)$ is a union of all the strata (with group $G$). By Theorem \ref{thm1} we see that the closure of any stratum is an irreducible Zariski closed subset of $(\overline{\mathfrak{M}_g}-\mathfrak{M}_g)(G)$. Therefore to understand the components of $(\overline{\mathfrak{M}_g}-\mathfrak{M}_g)(G)$, is equivalent to understanding the maximal strata.
\section{The maximal strata}
In the previous section we have interpreted the problem of determining irreducible components of $(\overline{\mathfrak{M}_g}-\mathfrak{M}_g)(G)$ into determining the maximal $G$-strata.

In this section we first discuss in general when a stratum is maximal. Then with certain additional conditions we give an explicit description via the associated combinatoric data.
\begin{defin}
 	 Given a stratum, we say that (the action of) $G$ is {\em maximal} if for any general curve $(C,G,\rho)$ inside the stratum, there is no subgroup $G'\subset Aut(C)$ isomorphic to $G$ such that the induced $G'$-marked stable curve $(C,G')$ is $G'$-equivariantly smoothable or the dimension of the stratum corresponding to $(C,G')$ is larger than the dimension of the given one. 
 \end{defin}
 It is clear that a stratum is maximal if and only if the corresponding action of $G$ is maximal.
\begin{defin}\label{StrataAut}
		The automorphism group of a stratum is defined to be the automorphism group of a general curve inside the stratum.
\end{defin} 
\begin{remark}
	Definition \ref{StrataAut} makes sense since for the stratum of a smooth $G$-marked curve, the general curves have isomorphic automorphism groups.
\end{remark}
It is easy to compute the automorphism group of a given stratum: pick a general curve $(C,G,\rho)$ in the stratum, write 
$$C=\sum_iC_i=\sum_{\lambda\in\Lambda}\sum^{s_\lambda}_{t=1}C_{\lambda,t},$$
where $\Lambda$ is the index set of isomorphism classes of the irreducible components with marked points $(C_i,\mathcal{N}^{(1)}_i,\mathcal{N}^{(2)}_i)$ and $s_\lambda$ is the number of curves $(C_i,\mathcal{N}^{(1)}_i,\mathcal{N}^{(2)}_i)$ belonging to the isomorphism class $\lambda$.\\

Clearly $Aut(C)$ is a subgroup of $\prod_{\lambda\in\Lambda}((\prod_{t=1}^{s_\lambda}Aut(C_{\lambda,t}))\rtimes \mathfrak{S}_{s_\lambda})$ consisting of elements preserve the nodes of $C$, where for each class $\lambda$ we fix an identification of $\Aut(C_{\lambda,t})$ for all curves $C_{\lambda,t}$ and  the semi-direct product is determined by the following group homomorphism:
$$\mathfrak{S}_{s_\lambda}\rightarrow Aut(\prod_{t=1}^{s_\lambda}Aut(C_{\lambda,t})), \sigma\mapsto \phi_\sigma:(g_1,...,g_{s_\lambda})\mapsto(g_{\sigma(1)},...,g_{\sigma(s_\lambda)}).$$
 Once we know the automorphism group of the stratum, we can find the subgroups which are isomorphic to $G$ and hence determine whether the stratum is maximal.  \\
Now for fixed genus $g$ and group $G$, we can determine the irreducible components (equivalently, the maximal strata) of $(\overline{\mathfrak{M}_g}-\mathfrak{M}_g)(G)$ since the possible configurations are finite. \\
However if we do not fix the genus $g$, due to the following phenomena, it is not so easy to obtain a brief description of the irreducible components even for cyclic groups.\\
Recall that in the smooth case a stratum with group $G$ is called {\em full} if the automorphism group of the stratum equals $G$. Now for a $G$-marked stable curve $(C=\sum_iC_i,G,\rho)$, if for some $i$, the group $H_i$ is not full for the induced action of $H_i$ on $\widetilde{C_i}$, the complicity of $Aut(C)$ increases.\\
We give an example for the automorphism group of a non-full stratum of smooth curves. By \cite{MSSV02},  Lemma 4.1 we know that for a general smooth curve $C$ inside a non-full stratum, $G$ is a normal subgroup of $\Aut(C)$ and $\Aut(C)/G$ is isomorphic to $\mathbb{Z}/2\mathbb{Z}$, $(\mathbb{Z}/2\mathbb{Z})^2$, etc. In the case of $G$ being cyclic and $Aut(C)/G\simeq (\mathbb{Z}/2\mathbb{Z})^2$, by \cite{MSSV02}, Lemma 4.1 there are three elements $b_1,b_2,b_3 \in Aut(C)-G$, such that $b_i$ has order 2 and the product $b_1b_2b_3$ is contained in $G$. The following proposition tells us in this case all the possibilities for $Aut(G)$: (\cite{Li16}, Lemma 5.7)
\begin{proposition}\label{nonfull}
		Let $G(H)$ be a group containing a normal cyclic subgroup $H$ of order $d$ such that $G(H)/H\simeq (\mathbb{Z}/2)^2$. Assume in addition that there exist three elements $b_1,b_2,b_3 \in G(H)-H$ such that $b_i$ has order 2 and the product $b_1b_2b_3$ is contained in $H$. Then $G(H)$ has the presentation:
		$$\{\alpha,\beta_1,\beta_2|\alpha^d=1,\beta_1^2=\beta_2^2=1,\beta_1\alpha=\alpha^{l_1}\beta_1,\beta_2\alpha=\alpha^{l_2}\beta_2,\beta_1 \beta_2=\beta_2\beta_1\alpha^{e_{1,2}}\}$$
		such that $0\leq l_1,l_2,e_{1,2}< d$, $\gcd(l_i,d)=1$, $l_i^2 \equiv 1 \mod d$, $d|(l_i+1)e_{1,2}$, for $i=1,2$ and $\gcd(d,l_1l_2+1)|e_{1,2}$. \\
		Moreover, $\gamma:=\bar{\alpha}$ is a generator of $H$; $b_i=\bar{\beta_i}$, $b_i\gamma b_i=\gamma^{l_i}$ for $i=1,2$ and $b_2b_1b_2=b_1\gamma^{e_{1,2}}$; $b_3=b_1b_2\gamma^f$, where $f$ is an integer such that $0\leq f<d$ and $d|((l_1l_2+1)f+e_{1,2})$.  
\end{proposition}
 In the smooth case, in order to determine if a stratum is maximal, we only need to compute the subgroups of $G(H)$ which are isomorphic to $H$. However, for stable curves, we have to compute all the cyclic subgroups of $\Aut(\widetilde{C_i})$ and solve a combinatoric problem concerning the dual graph and all the cyclic subgroups of $Aut(\widetilde{C_i})$ for each $C_i$.
 
In order to have a more detailed discussion, for the rest of the article we make the following assumptions:
\begin{assumption}\label{assum}
\begin{itemize}
\item[(0)] $(C=\sum_{i \in I}C_i,G,\rho)$ is $G$-equivariantly non-smoothable.\footnote{Of course, this is a necessary condition for having a maximal stratum.}
\item[(1)] For a general stable curve $(C,G,\rho)$ in the stratum we have $H_i=Aut(\widetilde{C_i})$ and $g(\widetilde{C_i})\geq 2$ for all $i$.
\item[(2)] For any $i\in I$, the parameterizing space $\mathcal{T}_{h_i,r_i}$ has dimension$>0$.
\end{itemize}
\end{assumption}

\begin{remark}
	\begin{itemize}
		\item[(1)] 
By assumption (2), for a general curve $(C,G,\rho)$ in the stratum, two irreducible components coming from different $G$-orbits must be non-isomorphic, hence we have $Orb=\Lambda$. Therefore any $\beta\in \Aut(C)$ fixes the $G$-orbits and induces $\beta_o:=\beta|_{C(o)} \in \Aut(C(o))$, conversely $(\beta_o)_{o\in Orb}$ determines $\beta$.  
\item[(2)] Given $\beta=(\beta_o)\in Aut(C)$, the order of $\beta$ is $l.c.m\{Ord(\beta_o)\}$. Using the isomorphism in Lemma \ref{canact} and regarding $\beta_o$ as an element in $(\prod_{j=1}^{n_o}Aut(C^{(j)}_o))\rtimes \mathfrak{S}_{n_o}$, we can write $\beta_o=((\beta_{o,1},...,\beta_{o,n_o}),\overline{\beta_o})$. What is the order of $\beta_o$? Assume that $\overline{\beta_o}$ has $\mu_o(\beta)$ orbits in $o$ with lengths $l_1,...,l_{\mu_o(\beta)}$, then we have $$\overline{\beta_o}=(i_{1},...,i_{l_1})(i_{l_1+1},...,i_{l_1+l_2})...(i_{n_o-l_{\mu_o(\beta)}+1},...,i_{n_o})$$
 and $$Ord(\beta_o)=l.c.m(Ord(\beta_{o,i_1}...\beta_{o,i_{l_1}})l_1,...,Ord(\beta_{o,(i_{n_o-l_{\mu_o(\beta)}}+1)}...\beta_{o,i_{n_o}})l_{\mu_o(\beta)}).$$
\end{itemize}
\end{remark}
We want to understand when the stratum is maximal. For this purpose we study first the quotient $C/\langle\beta\rangle$, where $\beta\in \Aut(C)-G$ is an element of order $d$. \\

\begin{lemma}\label{fac}
For any $\beta\in Aut(C)$, the quotient map $\pi:C\rightarrow C':=C/G$ factors through the quotient map $C\rightarrow C/\langle\beta\rangle$.
\end{lemma}
\begin{proof}
Note that we have the following decomposition of $C'$ into irreducible components: $C'=\sum_{o\in Orb} C'_o$.\footnote{For any $i\in o$, $\widetilde{C'_i}=$ normalization of $C'_o$}

For the lemma, it suffices to show that for any $P\in C$, $\pi(P)=\pi(\beta(P))$. By assumption $(2)$ we have that $\beta(C(o))=C(o)$, therefore it suffices to consider the map $\pi|_{C(o)}:C(o)\rightarrow C'_o$ and $\beta_o:=\beta|_{C(o)}$. Using Lemma \ref{canact} we see this is equivalent to considering the map $\pi_o:\bigsqcup^{n_o}_{j=1}C^{(j)}_o\rightarrow C'_o$, where $\pi_o$ is the composition of $\pi|_{C(o)}$ with the natural map $\bigsqcup^{n_o}_{j=1}C^{(j)}_o\rightarrow C(o)$.

 We determine first the fibre of $\pi_o$: noting that $\overline{\gamma}$ acts transitively on the vertices inside $o$, hence any fibre of $\pi_o$ must contain at least a point in $C^{(1)}_o$, say $x^{(1)}\in C^{(1)}_o$. Here we only discuss in detail the case where $x^{(1)}$ does not lie in the inverse image of a node of $C(o)$, the other case is similar. Then using the isomorphism of Lemma \ref{canact} we see that $\pi^{-1}_o(\pi_o(x^{(1)}))=\{x^{(1)},\gamma^{n_o}(x^{(1)}),...,\gamma^{n''_o-n_o}(x^{(1)});$
$x^{(2)},...,\gamma^{n''_o-n_o}(x^{(2)});...;x^{(n_o)},...,\gamma^{n''_o-n_o}(x^{(n_o)})\}$, where $n''_o:=d/|G''_i|$ for any $i\in o$ and $x^{(j)}$ denotes the point on $C^{(j)}_o$ which equals to $x^{(1)}$ via the identification $C^{(j)}_o=C^{(1)}_o$. Now since $\beta_o=((\beta_{o,1},...,\beta_{o,n_o}),\overline{\beta_o})$ and $\forall x^{(j)}\in C^{(j)}_o$, $\beta_o(x^{(j)})=\beta_{o,\overline{\beta_o}(j)}x^{(\overline{\beta_o}(j))}$, by assumption $(1)$ we have that $\beta_{o,\overline{\beta_o}(j)}\in \langle \gamma^{n_o}\rangle$ and hence $\beta_o(x^{(j)})\in \pi^{-1}_o(\pi_o(x^{(1)}))$. 
\end{proof}
\begin{remark}\label{Two cases}
For simplicity we denote by $\mathfrak{M}_{C'}$ the stratum corresponding to $C\rightarrow C'$ and by $\mathfrak{M}_{\beta}$ the stratum corresponding to $C\rightarrow C/\langle\beta \rangle$ (similarly for $\mathcal{T}_{C'}$ and $\mathcal{T}_\beta$), Lemma \ref{fac} says that $\mathfrak{M}_{C'}$ is contained in $\mathfrak{M}_{\beta}$ (not just in $\overline{\mathfrak{M}_{\beta}}$!). If $\mathfrak{M}_{C'}$ is not maximal, there are two cases:\\
(1) there exists a $\beta\in Aut(C)$ of order $d$, such that $dim \mathfrak{M}_{\beta}>dim \mathfrak{M}_{C'}$;\\
(2) there exists a $\beta\in Aut(C)$ of order $d$, such that $dim \mathfrak{M}_{\beta}=dim\mathfrak{M}_{C'}$ and $(C,\beta)$ is $G$-equivariantly smoothable.
\end{remark}
Recall that $\mathcal{T}_{C'}$ is the product of the parameterizing spaces of all the coverings $C(o)\rightarrow C'_o$, which is isomorphic to the parameterizing space $\mathcal{T}_{C'_o}$ of the covering $C^{(1)}_o\rightarrow C'_o$ (strictly speaking, of the covering $\widetilde{C^{(1)}_o}\rightarrow \widetilde{C'_o}$). Denoting by $(C/\langle\beta\rangle)(o)$ the inverse image of $C'_o$ in $C/\langle\beta\rangle$, the number of irreducible components of $(C/\langle\beta\rangle)(o)$ is $\mu_o(\beta)$. Hence the parameterizing space of $C(o)\rightarrow (C/{\beta})(o)$ is a product of $\mu_o(\beta)$ parameterizing spaces of the irreducible components of $(C/{\beta})(o)$, all of which have dimension greater or equal to $dim\mathcal{T}_{C'_o}$. Now we can characterize case $(1)$ of Remark \ref{Two cases}:
\begin{lemma} \label{case 1}
Case (1) of Remark \ref{Two cases} happens if and only if $\exists o\in Orb$, such that one of the following cases occurs:
\begin{itemize}
\item[(a)] $\mu_o(\beta)>1$;
\item[(b)] $\mu_o(\beta)=1$ and $3Ord(\beta_o^{n_o})\leq Ord(H_i)$ for any $i\in o$; 
\item[(c)] $\mu_o(\beta)=1$, $2Ord(\beta_o^{n_o})= Ord(H_i)$ for any $i\in o$ and we are not in the exceptional cases of Proposition \ref{exception}.
\end{itemize}

\end{lemma}
\begin{proof}
First note that $Ord(H_i)$ does not depend on the choice of $i\in o$.\\
By assumption (1) $dim\mathcal{T}_{C'_o}\geq 1$, hence (a) follows from the preceding discussion. If $\mu_o(\beta)=1$ for all $o$, we have that $C/\langle {\beta}\rangle(o)$ is an irreducible curve, in case $(b)$ by Proposition \ref{exception} we always have $dim\mathcal{T}_{C/\langle {\beta}\rangle(o)}>dim\mathcal{T}_{C'_o}$; if $2Ord(\beta_o^{n_o})= Ord(H_i)$ then $dim\mathcal{T}_{(C/\langle\beta\rangle)(o)}>dim\mathcal{T}_{C'_o}$ except for the cases in Proposition \ref{exception}.
\end{proof}
Now we discuss briefly the exceptional cases mentioned in \ref{case 1}, since the technique we use here is independent of the other part of this paper, we only sketch the proof and for reader who are interested in details we refer to \cite{MSSV02}, \cite{CLP11} and \cite{LW16}.
\begin{proposition}\label{exception}
	Given an admissible branching sequence $[(k_1,...,k_{d-1})]$ for $g\geq 2$ such that, $dim\mathcal{T}_{g;d,[(k_1,...,k_{d-1})]}\geq 1 $ and for any general curve $C\in \mathcal{T}_{g;d,[(k_1,...,k_{d-1})]}$, $G=Aut(C)$. For any proper subgroup $G'$ of $G$, we have an induced cyclic cover of degree $d':=order(G')$: $C\rightarrow C/G'$ and hence an admissible sequence $[(k'_1,k'_2,...,k'_{d'-1})]$ for $d'$ and $g$. Then $dim_\mathbb{C}\mathcal{T}_{g;d,[(k_1,...,k_{d-1})]}>dim_\mathbb{C} \mathcal{T}_{g;d',[(k'_1,...,k'_{d'-1})]}$ except for two cases:\\
	$(1)$ $d=2d'$, $2\mid d'$, $C/G\simeq \mathbb{P}^1$, $[(k_1,...,k_{d-1})]=[(1,0,...,0,k_{d'}=2,0,...,0,1)]$ and $[(k'_1,...,k'_{d'-1})]=[(1,0,...,0,k_{d'/2}=2,0,...,0,1)]$;\\
	$(2)$ $d=2d'$, $2\nmid d'$, $C/G\simeq \mathbb{P}^1$, $[(k_1,...,k_{d-1})]=[(0,1,0,...,0,k_{d'}=2,0,...,0,1,0)]$ and $[(k'_1,...,k'_{d'-1})]=[(1,0,..,0,1)]$.
\end{proposition}
\begin{proof}
	If $dim_\mathbb{C}\mathcal{T}_{g;d,[(k_1,...,k_{d-1})]}=dim_\mathbb{C} \mathcal{T}_{g;d',[(k'_1,...,k'_{d'-1})]}$, then we have that $\mathcal{T}_{g;d,[(k_1,...,k_{d-1})]}=\mathcal{T}_{g;d',[(k'_1,...,k'_{d'-1})]}$. Now since the condition of \cite{MSSV02}, Lemma 4.1 is satisfied, we see that the pair $(C,G')$ must be one of the cases there.\\
	By assumption $dim_\mathbb{C} \mathcal{T}_{g;d',[(k'_1,...,k'_{d'-1})]}\geq 1$, hence only the cases $I,II,III$ in \cite{MSSV02}, Lemma 4.1 happen.\\
	Case $III-c$ is excluded since here $Aut(C)$ is a cyclic group, which can not have a quotient group isomorphic to $(\mathbb{Z}/2)^2$.\\
	For the remaining cases, we have $d=2d'$ and $C/Aut(C)\simeq \mathbb{P}^1$.
	The Cases $I$, $II$, and $III-a$ are excluded since in these cases the conditions $(1)$ and $(2)$ of Definition \ref{adm} can not be simultaneously satisfied.\\
	For case $III-b$, $C\rightarrow C/Aut(C)$ has four branching points $P_1,P_2,P_3,P_4$ with branching sequence $(2,2,c_3,c_4)$ such that $2<c_3\leq c_4$. For $1\leq j \leq 4$, let  $i_j$ be the index such that $P_j\in D_{i_j},$ we have that $d/\gcd(d,i_1)=d/\gcd(d,i_2)=2$, $d/\gcd(i_3,d)=c_3$ and $d/\gcd(i_4,d)=c_4$. Hence we get that $i_1=i_2=d'$. Moreover $i_1,i_2,i_3,i_4$ should satisfy conditions $(1)$ and $(2)$ of Definition \ref{adm}: Condition (1) says that $d|(d'+d'+i_3+i_4)$, which implies that $d|(c_3+c_4)$;
	Condition (2) says that $\gcd(d,d',c_3,c_4)=1$. There are two possibilities: $$[(k_1,...,k_{d-1})]=[(1,0,...,0,k_{d'}=2,0,...,0,1)]$$ or $$2\nmid d',\ [(k_1,...,k_{d-1})]=[(0,1,0,...,0,k_{d'}=2,0,...,0,1,0)].$$ 
	Noting that there is one more restriction in case $III-b$ that $C/G'\rightarrow \mathbb{P}^1$ is a double cover branched in two points on $\mathbb{P}^1$, we see that if $$[(k_1,...,k_{d-1})]=[(1,0,...,0,k_{d'}=2,0,...,0,1)],$$ then we must have $2|d'$, otherwise $C/G'\rightarrow \mathbb{P}^1$ is branched on four points on $\mathbb{P}^1$; in this case $$[(k'_1,...,k'_{d'-1})]=[(1,0,...,0,k_{d'/2}=2,0,...,0,1)].$$
	 For the other case where $2\nmid d'$ and  $$[(k_1,...,k_{d-1})]=[(0,1,0,...,0,k_{d'}=2,0,...,0,1,0)],$$ we get $[(k'_1,...,k'_{d'-1})]=[(1,0,..,0,1)]$.
\end{proof}  
\begin{remark}
Let us look at the case when $dim\mathcal{T}_{\beta}=dim \mathcal{T}_{C'}$: first $\mu_o(\beta)=1$ for all $o\in Orb$, which means that $\overline{\beta_o}$ acts transitively on the vertices in $o$. 

For the subcurve $C/(\beta)(o)$, if we are not in the exceptional cases, then we have $(C/\langle\beta\rangle)(o)\simeq C'_o$ and $\beta_o^{n_{o}}|_{C^{(1)}_o}=\prod^{n_o}_{k=1}\beta_{o,k}$ is a generator of $Aut(C^{(1)}_o)$.  Otherwise $C'_o$ is rational and $C/\langle {\beta}\rangle (o)\rightarrow C'_o$ is a double cover and $\langle \prod^{n_o}_{k=1}\beta_{o,k} \rangle $ is the (unique) index $2$ subgroup of $Aut(C^{(1)}_o)$ which arises from the exceptional cases in Proposition \ref{exception}.
\end{remark}
What remains to determine is that under the condition $dim\mathcal{T}_{\beta }=dim \mathcal{T}_{C'}$, when is $(C,\beta)$ $G$-equivariantly non-smoothable? 
Here we apply Proposition \ref{nonsm} to a node $p\in \mathcal{N}_{i_1}\cup\mathcal{N}_{i_2}$. We have two cases: $i_1=i_2$ and $i_1\neq i_2$, which we treat separately.\\

\smallskip
Case (I)\\
If $i_1=i_2=:i$, we must have $G_p\subset G_i$. Denoting by $\{p_1,p_2\}$ the inverse image of $p$ of the normalization map $\widetilde{C_i}\rightarrow C_i$, we have the following easy lemma:
\begin{lemma}\label{loopnode}
   For any $g\in H_i$, (regarding $g$ also as an automorphism of $\widetilde{C_i}$,) there are three possibilities:
	\begin{itemize}
		\item[(a)] $g(p_i)=p_i$ for $i=1,2$.
		\item[(b)] $g(p_1)=p_2$ and $g(p_2)=p_1$.
		\item[(c)] $g(p_1)=p'_1$ and $g(p_2)=p'_2$, where $\{p_1',p_2'\}$ is the inverse image of $g(p)(\neq p)$.
	\end{itemize}
\end{lemma}
\begin{proof}
	Obvious.
\end{proof}
We apply Proposition \ref{nonsm} to $(C_i,H_i=\langle\overline{\gamma^{n_i}|_{C_i}} \rangle)$ and $(C_i,\langle \overline{\beta^{n_i}|_{C_i}} \rangle)$. If $\langle \overline{\beta^{n_i}|_{C_i}} \rangle=H_i$, then it is clear that $p$ is non-smoothable for $(C_i,\langle \overline{\beta^{n_i}|_{C_i}} \rangle)$ and hence non-smoothable for $(C,\langle\beta \rangle)$. \\
Assume we are in the exceptional cases of Proposition \ref{exception}. First we consider exceptional case (1), recall that $\pi_i:\widetilde{C_i}\rightarrow \widetilde{C_i}/H_i$ is branched on four points $P_1,...,P_4$ with branching indices $(2,2,d_i,d_i)$ with $d_i\geq 4$. Note that $p_1$ or $p_2$ does not belong to either $\pi^{-1}_i(P_1)$ or $\pi^{-1}_i(P_2)$:

 Case (a) in Lemma \ref{loopnode} does not occur since $\#\pi^{-1}_i(P_1)=\#\pi^{-1}_i(P_2)=d_i/2\geq 2$. 
 
 For the same reason, if $d_i\geq 6$, then Case (b) does not occur; if $d_i=4$, Case (b) does not occur, either, otherwise $p$ is $H_i$-equivariantly smoothable. 
 
 Case (c) does not occur, otherwise $p$ is $H_i$-equivariantly smoothable.
\medskip  

Hence we may assume $p_1=\pi^{-1}_i(P_3)$ and $p_2=\pi^{-1}_i(P_4)$. Let $z_j$ be a local coordinate near $p_j$, j=1,2, the action of $H_i$ near $p$ is
 $$\overline{\gamma_i^{n_i}}:z_1\mapsto \zeta_{d_i}z_1,\ z_2\mapsto \zeta_{d_i}^{-1}z_2.$$ 
 This implies that $p$ is $H_i$-equivariantly smoothable, a contradiction. Therefore we see that exceptional case $(1)$ does not occur.

Using a similar argument we see that exceptional case (2) does not occur, either.\\

Case (II) $i_1\neq i_2$.\\
 We have two subcases:
\begin{itemize}
	\item[(i)] $G_p$ fixes $i_1$ and $i_2$ respectively.
	\item[(ii)] $G_p$ exchanges  $i_1$ with $i_2$.
\end{itemize}
Subcase (i):\\
 Let $x$ (resp. $y$) be a local parameter on $C_{i_1}$ (resp. on $C_{i_2}$) near $p$. Denoting by $a_l$ the smallest positive integer such that $\gamma^{n_{i_l}a_l}(p)=p$ for $l=1,2$ (note that we necessarily have $ n_{i_1}a_1=n_{i_2}a_2$), then locally the action of $\overline{\gamma^{n_{i_l}a_l}}$ around $p$ is given by $(x,y)\mapsto (\zeta_{p,1}^{b_1}x,\zeta_{p,2}^{b_2}y)$ for some natural numbers $b_1,b_2$, where $\zeta_{p,l}$ is a primitive $n''_{i_l}/(n_{i_l}a_l)$-th root of unity and $n''_{i_l}:=d/|G''_{i_l}|$ (we require that $b_1,b_2\leq n''_{i_l}/(n_{i_l}a_l)$). The condition that $(C,G,\rho)$ is non-smoothable implies that $\zeta_{p,1}^{b_1}\zeta_{p,2}^{b_2}\neq 1$. 

By our assumption $(1)$ we have $\overline{\beta^{n_{i_l}}|_{C_{i_l}}}\in \Aut(C_{i_l})=\langle\overline{\gamma^{n_{i_l}}|_{C_{i_l}}}\rangle$, hence we get that $\overline{\beta^{n_{i_l}}|_{C_{i_l}}}=(\overline{\gamma^{n_{i_l}}})^{c_l}$ for some $0\leq c_l< Ord(H_{i_l})$. By Proposition \ref{exception} we have two possibilities: 
\begin{itemize}
	\item  $\langle \overline{\beta^{n_{i_l}}|_{C_{i_l}}}\rangle=\Aut(C_{i_l})$, which is equivalent to $\gcd(c_l,Ord(H_{i_l}))=1$.
	\item  We are in the exceptional cases where $\langle \overline{\beta^{n_{i_l}}|_{C_{i_l}}}\rangle$ is the index 2 subgroup of $\Aut(C_{i_l})$, and $c_l=2c'_l$ for some $0\leq c'_l<Ord(H_{i_l})/2$ with $\gcd(c'_l,Ord(H_{i_l}))=1$.
\end{itemize}
   The action of $\overline{\beta^{n_{i_l}a_l}}$ is given by $(x,y)\rightarrow (\zeta_{p,1}^{b_1c_1}x,\zeta_{p,2}^{b_2c_2}y)$. We see easily that $p$ is non-smoothable for $(C_i,\langle \overline{\beta^{n_i}|_{C_i} }\rangle)$ iff $\zeta_{p,1}^{b_1c_1}\zeta_{p,2}^{b_2c_2}\neq 1$.\\
Subcase (ii):\\
Observe that $i_1$ and $i_2$ lie in the same orbit, hence we have $d_{i_1}=d_{i_2}=2Ord(G_p)$. The situation here is similar to that of case (I), we get that $p$ is non-smoothable for $(C_i,\langle \overline{\beta^{n_i}|_{C_i} }\rangle)$ and leave the details to the reader.\\

Now we fix the local parameters for each nodes where subcase $(II-i)$ happens, then we obtain an unordered pair $(\zeta_{p,1}^{b(p,1)},\zeta_{p.2}^{b(p,2)})$ at each node $p$ which is determined by $\gamma$. 
For any $\beta\in \Aut(C)$ with degree $d$ such that $dim\mathcal{T}_{\beta}=dim \mathcal{T}_{C'}$, we get a pair of integers $(c(\beta,p,1),c(\beta,p,2))$ at each node.

 Combining with the previous argument, we obtain our main theorem:
\begin{theorem}\label{mainthm}
		 Under the conditions of Assumption \ref{assum}, we have the following:
		\begin{itemize}
			\item[(1)] For a $G$-equvariantly non-smoothable $G$-marked stable curve $(C=\sum_{i\in I}C_i,G,\rho)$, the induced stratum $\mathfrak{M}_{C'}$, where $C'=C/G$, is maximal iff for a general stable curve (by abuse of notation we denote still by) $(C,G,\rho)$ in the stratum:
			\begin{itemize}
				\item[(a)] The cases in Lemma \ref{case 1} do not occur.
				\item[(b)]  For any $\beta\in \Aut(C)$ (of order $d$) and any node $p$ where Case (II-i) happens, the following holds:
				$$\zeta_{p,1}^{b(p,1)c(\beta,p,1)}\zeta_{p,2}^{b(p,2)c(\beta,p,2)}\neq 1.$$
			\end{itemize}
			\item[(2)]The Zariski closure of each maximal stratum in $(1)$ is an irreducible component of $(\overline{\mathfrak{M}_g}-\mathfrak{M}_g)(G)$.
		\end{itemize}
	\end{theorem}
\begin{remark}
	The above argument also shows that, unlike in the smooth case (cf. \cite[Theorem 1]{Cor87} and \cite[Theorem 3.4]{Cat12}), for stable curves, usually a component corresponds to more than one numerical types.
\end{remark}
\section*{Acknowlegement}
The author is currently supported by Shanghai Center for Mathematical Science (SCMS), Fudan University. Part of this work took place in the framework of the ERC Advanced grant n. 340258, `TADMICAMT'.

The author would like to thank Fabrizio Catanese for suggesting this topic and for many helpful discussions.
\bibliographystyle{alpha}

\end{document}